\documentclass[11pt,reqno]{amsart}
\usepackage{amsmath,amsxtra,amssymb,latexsym, amscd,amsthm}
\usepackage[unicode]{hyperref}
\usepackage[top=2cm, bottom=2.2cm, left=2cm, right=2cm]{geometry}
\usepackage[mathscr]{eucal}
\usepackage{epsfig, amssymb,amsmath,float}
\usepackage{multirow}
\usepackage[norule]{footmisc}
\newtheorem{definition}{Definition}[section]
\newtheorem{lemma}{Lemma}[section]

\newtheorem{theorem}{Theorem}[section]
\newtheorem{corollary}{Corollary}[section]
\newtheorem{remark}{Remark}[section]

\newtheorem{proposition}{Proposition}

\def\Z{{\mathbb Z}}
\title{On spectral simple modules over Leavitt path algebras}

\author{Nguyen Bich Van}
\begin{document}
\maketitle

\begin{abstract}
In this work we describe all simple spectral modules over Leavitt path algebras as induced modules from irreducible representations of the isotropy groups.
\end{abstract}
\section {Introduction}
In 2012 X.W. Chen \cite{Chen} and in 2014 P. Ara and K.M. Rangaswamy \cite{AKR} constructed some classes of simple modules over Leavitt path algebras.  In this work we try to use Steinberg algebras to describe these modules as induced modules from irreducible representations of the isotropy groups.
For convenience of the readers at first we recall some useful facts from \cite{Chen}, \cite{St} and \cite{AKR}.
\subsection{Induced representations of groupoid algebras}

Let $G$ be an ample groupoid, $K$ be a field, $A_K(G)$ be the groupoid algebra of $G$. For each element $x\in G^{(0)}$ we define
\begin{definition} 
The isotropy group of $x\in G^{(0)}$ is
$$G_x=\{\alpha\in G|s(\alpha)=r(\alpha)=x\}.$$
\end{definition}
\begin{definition}
The orbit of $x\in G^{(0)}$ is
$$O_x=\{y\in G^{(0)}|\exists \alpha\in G: s(\alpha)=y, r(\alpha)=x\}$$
\end{definition}
\begin{remark}
If the orbits of $x$ and $y$ are the same, then $G_x\simeq G_y$
\end{remark}
Let $L_x=\{\alpha\in G|r(\alpha)=x\}$ and $KL_x$ be the vector space with the basis $L_x$. Consider the action of $G_x$ on $L_x$ by right multiplication. This action induces a right $KG_x$-module structure on $KL_x$. 

The left $A_K(G)$-module structure on $KL_x$ is defined by:
\begin{equation}
\label{schuz}
ft=\sum_{y\in L_x}f(yt^{-1})y\mbox{ }\forall f\in A_K(G),\forall t\in L_x.
\end{equation}
\begin{proposition} (\cite{St})\label{char_act} We have
\begin{equation}
1_{B}t=\begin{cases}
Bt\quad\text{if}\ s(t)\in B^{-1}B\\
0\quad\text{else}
\end{cases}
\end{equation}
for all compact bisections $B$ of $G$, for all $t\in L_x$. Consequently, $KL_x$ is a well-defined $A_K(G)$-$KG_x$ bimodule.
\end{proposition}
\begin{definition}
Let $x\in G^{(0)}$, $V$ be a $KG_x$-module. The corresponding induced $A_K(G)$-module is $Ind_x(V)=KL_x\otimes_{KG_x} V$.
\end{definition}
\begin{theorem} \label{simp}If $V$ is a simple $KG_x$-module, then $Ind_x(V)$ is a simple $A_K(G)$-module.
\end{theorem}
\begin{remark}\label{isom}
If $x,y\in G^{(0)}$ are in the same orbit,i.e. $\exists \alpha\in G: s(\alpha)=y,r(\alpha)=x$ we know that $G_x\simeq G_y$. Let $V$ be a $KG_x$-module, then $V$ can be made into a $KG_y$-module by setting  $\gamma u=(\alpha^{-1}\gamma\alpha)u$, where $\gamma\in G_y, u\in V$. Then $Ind_x(V)\simeq Ind_y(V)$ via the map $t\otimes u\mapsto t\otimes t\alpha^{-1}\otimes u$ for $t\in L_x$ and $u\in V$.
\end{remark}
\subsection{Restricted modules}
\begin{definition}(\cite{St}) Let $W$ be an $A_K(G)$-module.
For $x\in G^{(0)}$ we define 
$$N_x=\{U \text{is a compact open subset of } G^{(0)}|x\in U\}$$ and $Res_x(W):=\cap_{U\in N_x}1_UW$.
\end{definition}
\begin{proposition}(\cite{St}) $Res_x(W)$ becomes a $KG_x$-module under the  following action:  for $t\in G_x, w\in Res_x(W)$
 $tw=1_Bw$, where $B$ is a compact bisection of $G$ s.t. $t\in B$.
\end{proposition}
\begin{proposition}\label{res_in}\cite{St}
Let $V$ be a $KG_x$-module. Then $Res_xInd_x(V)\simeq x\otimes V\simeq V$ as $KG_x$-modules.
\end{proposition}
\begin{lemma} (\cite{St})
If $W$ is a simple $A_K(G)$-module, then for any $x\in G^{(0)}$ $Res_x(W)$ is either zero or a simple $KG_x$-module.
\end{lemma}
\begin{definition} [Spectral module]
Let $W$ be an $A_K(G)$-module. $W$ is called spectral if $\exists x\in G^{(0)}$ so that $Res_x(W)\neq 0$
\end{definition}
\begin{lemma}\cite{St}\label{noniso}
If $x,y\in G^{(0)}$ have distinct orbits, then induced modules of the form $Ind_x(V)$ and $Ind_y(V)$ are not isomorphic.
\end{lemma}
\begin{theorem}\label{specmod}(\cite{St})

Let $G$ be an ample groupoid and $D\subset G^{(0)}$ containing exactly one element from each orbit. Then there is a bijection between spectral simple $A_K(G)$-modules and pairs $(x,V)$, where $x\in D$, $V$ is a simple $KG_x$-module (taken up to isomorphism). The corresponding module is $Ind_x(V)$ . The finite dimensional simple $A_K(G)$-modules are spectral modules and correspond to those pais $(x,V)$, where the orbit of $x$ is finite and $V$ is a finite dimensional simple $KG_x$-module.

\end{theorem}
 In the next section we will use this theorem to classify all spectral simple modules (in particular,  finite dimensional modules) over Leavitt path algebra.
 Let $E=(E^0, E^1,s,r)$ be a digraph, $K$ be a field. 
\begin{definition}\label{path}
A finite path of length $n$ is $p=e_1e_2...e_n$, where $e_1,...,e_n\in E^1,r(e_i)=s(e_{i+1}),\quad i=1,2,...,n-1.$ The length of $p$ is denoted by $|p|$, $s(p):=s(e_1), r(p:=r(e_n).$. By convention a vertex $v\in E^0$ is called a trivial path. A vertex is a path of length $0$ and has $v$  as the source and the range.

A nontrivial finite path $p=e_1...e_n$  is called closed if $r(e_n)=s(e_1)$. 
 A closed path $p$ is called simple if $p\neq c^m,$ where $c\neq p$ is a closed path.

An infinite path is $q=e_1e_2...$, where $r(e_i)=s(e_{i+1}),\quad \forall i$. We denote by $F(E)$ the set of all finite paths, and by $E^\infty$ the set of all infinite paths. 
\end{definition}

Let $X=E^\infty\cup \{p\in F(E)|r(p)\quad \text{is not regular}\}$.
\begin{definition}\label{tail}
Two paths $x,y\in X$ are called tail-equivalent if $\exists \mu,\nu\in F(E), x'\in X$ such that $x=\mu x', y=\nu x'$. We write $x\sim y$.

This is an equivalence relation and we denote by $[x]$ the equivalence class of all paths which are tail-equivalent to $x$. 
\end{definition}
\begin{remark}
By Definition \ref{tail} and the convention in Definition \ref{path} it is easy to see that two finite paths $x,y\in X$ are tail-equivalent if and only if $r(x)=r(y)=w$, which is not regular. In this case $[x]=[y]=[w]$.
\end{remark}
\begin{remark}
For any closed path $d$, we can write $d=c^m,$ where $c$ is a simple closed path. Then $d^\infty=c^\infty$.
\end{remark}
\begin{definition}
An infinite path is called rational if it is tail-equivalent to $c^\infty$, where $c$ is a simple closed path, and called irrational otherwise. We denote by $E^\infty_{rat.}$ the set of all rational paths and by $E^\infty_{irrat.}$ the set of all irrational paths.
\end{definition}

\subsection{Chen simple modules}
Let us recall the construction of Chen simple modules over Leavitt path algebra in \cite{Chen} and \cite{R1}.
For every $x\in X,$ let $V_{[x],K}=\oplus_{p\in [x]}Kp$ as vector spaces. $V_{[x],K}$ is made a left 
$L_K(E)$-module by defining, for all $p\in [x], v\in E^0, e\in E^1:$

$v.p=\delta_{v,s(p)}p,$
$ e.p=\delta_{r(e),s(p)}ep,$

$e^*.p=\begin{cases}
p' \quad\text{if } p=ep'\\
0\quad \text{else}\end{cases}$

\begin{remark} 
From the construction of $V_{[x],K}$ it is easy to see that for all $\mu,\nu\in F(E),p\in [x]$ we have
\begin{equation}\label{chen_act}(\mu\nu^*).p=\begin{cases}\mu p'\quad\text{if }p=\nu p',\\
0\quad \text{else}\end{cases}\end{equation}
\end{remark}
In \cite{Chen} the twisted Chen modules are also defined. Let $\underline{a}=(a_e)_{e\in E^1}\in (K^*)^{E^1}$. 
We consider the automorphism $\sigma_{\underline{a}}$ of $L_K(E)$ given by $\sigma_{\underline{a}}(v)=v, \sigma_{\underline{a}}(e)=a_ee,\sigma_{\underline{a}}(e^*)=a_e^{-1}e^*$ for all $v\in E^0, e\in E^1.$ We have the $\underline{a}$-twisted Chen module $V_{[x],K}^{\underline{a}}$ as follows:
$V_{[x],K}^{\underline{a}}=\oplus_{p\in [x]}Kp$ as vector spaces and the action of $L_K(E)$ is given by: 
$L_K(E)\times V_{[x],K}^{\underline{a}}\rightarrow V_{[x],K}^{\underline{a}},$
$(z,p)\mapsto \sigma_{\underline{a}}(z).p,$ where $z\in L_K(E), p\in [x]$
\begin{remark}\label{rem2} $V_{[x],K}^{(1)_{e\in E^1}}= V_{[x],K}$\end{remark}
\begin{definition}
For each finite nontrivial path $q=e_1...e_n$, we denote by $a_q$ the product $a_{e_1}...a_{e_n}$. The element $\underline{a}$ is called $q$-stable if $a_q=1.$
\end{definition}
\section{Simple modules over Leavitt path algebras}

The groupoid of $E$ is 
$$G_E=\{(x,k,y)|\exists \mu,\nu\in F(E), x'\in X: x=\mu x', y=\nu x',|\mu|-|\nu|=k\}.$$
The multiplication is defined by
$$(x,k,y)(y,l,z)=(x,k+l,z).$$

The inversion is $$(x,k,y)^{-1}=(y,-k,x).$$

The unit space of $G_E$ is $G_E^{(0)}=\{(x,0,x)|x\in X\}$, which can be identified with $X$.

For each pair $(\mu,\nu)\in F(E)^2. r(\mu)=r(\nu)$ we define
$Z_{(\mu,\nu)}:=\{(\mu x,|\mu|-|\nu|,\nu x)|x\in X\}.$

These sets are compact bisections of $G_E$ and form a topological basis for $G_E$ 

The corresponding groupoid algebra is

$$A_K(G_E)=span_K\{1_{Z(\mu,\nu)}|\mu,\nu\in F(E), r(\mu)=r(\nu\}$$

The Leavitt path algebra $L_K(E)$ is isomorphic to $A_K(G_E)$ via the map 

$\pi: L_K(E)\rightarrow A_K(G_E)$
given by $\pi(v)=Z_{(v,v)}\quad \forall v\in E^0, \pi(e)=Z_{(e,r(e)},\pi(e^*)=Z_{(r(e),e)}\quad \forall e\in E^1.$

Each $L_K(E)$-module $M$ becomes an $A_K(G_E)$-module by setting $1_{Z_{(\mu,\nu)}}m=(\mu\nu^*)m,\quad \forall \mu,\nu\in F(E),\forall m\in M.$

From Theorem \ref{specmod} we know that every spectral simple module over $L_K(E)\simeq A_K(G_E)$ has the form $Ind_x(V),$ where $x\in G_E^{(0)}$ and $V$ is a simple $KG_x$-module. In the case of $L_K(E)$ the isotropy groups are clear

\begin{lemma}\label{iso}
\begin{enumerate}
\item If $x\in X\setminus E^\infty_{rat.}$, then $G_x$ is trivial.
\item If $x\sim c^\infty$, where $c$ is a simple closed path of length $n$, then $G_x\simeq n\Z.$  
\end{enumerate}
\end{lemma}
\begin{proof}
$G_x=r^{-1}(x)\cap s^{-1}(x)=\{(x,k,x)\in G_E\}.$ If $G_x$ is nontrivial, then $\exists k\neq 0$ such that $(x,k,x)\in G_E\implies \exists \mu,\nu\in F(E), p\in X$ for which $x=\mu p=\nu p, |\mu|-|\nu|=k$. By taking the inverse element we may suppose $\mu=\nu q$ for some $q\in F(E)$. Then we have $p=qp=qq...=q^\infty=c^\infty$ for some simple closed path $c$. Let $|c|=n$. We have $|\mu|-|\nu|=|q|=l|c|=ln\implies G_x\simeq n\Z.$
\end{proof}
\begin{remark}
$L_x=r^{-1}(x)=\{(y,l,x)|\exists \mu, \nu\in F(E), p\in X: y=\mu p, x=\nu p, |\mu|-|\nu|=l\}$\end{remark}
Now we are ready to describe all spectral simple modules over Leavitt path algebras by using Theorem \ref{specmod}.
\subsection{The case of irrational paths and finite paths ending at sinks or infinite emitters}

\begin{lemma}\label{lem_well}
Let $x\in X\setminus E^\infty_{rat.}$ and $y\in [x]$. Let $\underline{a}\in (K^*)^{E_1}$. If $\mu_1,\mu_2,\nu_1,\nu_2\in F(E), p_1,p_2\in X\setminus E^\infty_{rat.}$ and $y=\mu_1p_1=\mu_2p_2, x=\nu_1p_1=\nu_2p_2$, then $a_{\mu_1}a_{\nu_1}^{-1}=a_{\mu_2}a_{\nu_2}^{-1}$ and $|\mu_1|-|\nu_1|=|\mu_2|-|\nu_2|.$
\end{lemma}
\begin{proof}
We may suppose $\mu_1=\mu_2q, q\in F(E)$. Then we have \begin{equation}\label{eq1}\mu_2qp_1=\mu_2qp_2\Rightarrow p_2=qp_1\Rightarrow \nu_1p_1=\nu_2qp_1.
\end{equation}
 From \eqref{eq1} we see that there are 2 possibilities:
 \begin{enumerate}
 \item

 $\nu_1=\nu_2q\implies a_{\mu_1}a_{\nu_1}^{-1}=a_{\mu_2q}a_{\nu_2q}^{-1}
 =a_{\mu_2}a_{q}a_{\nu_2}^{-1}a_q^{-1}=a_{\mu_2}a_{\nu_2}^{-1}.$ 
On the other hand, $|\mu_1|-|\mu_2|=|\nu_1|-|\nu_2|=|q|.$

\item $p_1=qp_1=qq...=q^{\infty}, $ where $q$ is a closed path. But $p_1$ is not rational, so this case is impossible.
 \end{enumerate}\end{proof} 
  If $x\in X \setminus E^{\infty}_{rat.}$, by Lemma \ref{iso} $G_x$ is trivial: $G_x=\{(x,0,x)\}\implies KG_x\simeq K.$ 
The unique simple $K$-module is $K$, on which $K$ acts by left multiplication. .

\begin{proposition}\label{triv}
If $x\in X \setminus E^{\infty}_{rat.}$,  then for all $\underline{a}\in (K^*)^{E^1}$ we have
$Ind_x(K)\simeq V^{\underline{a}}_{[x],K} $ as $A_K(G_E)$-modules.
\end{proposition}
\begin{proof} 
$Ind_x(K)=KL_x\otimes_K K\simeq KL_x$ as $A_K(G_E)$-modules via the map $(y,l,x)\otimes k\mapsto k(y,l,x).$
 
 Consider the $K$-linear map $\varphi: KL_x\rightarrow V^{\underline{a}}_{[x],K}$ given by
 $\varphi((x,0,x))=x, \varphi ((y,l,x))=a_{\mu}a_{\nu}^{-1}y,$ where $\mu,\nu\in F(E)$ s.t. $x=\nu p, y=\mu p, |\mu|-|\nu|=l.$

 The $K$-linear map $\psi:V^{\underline{a}}_{[x],K}\rightarrow KL_x $ is given by $\psi(y)=a_{\mu}^{-1}a_{\nu}(y,l,x),$ where $\mu,\nu\in F(E) \quad \text{s.t.} y=\mu p,x=\nu p,|\mu|-|\nu|=l.$
  By Lemma \ref{lem_well} $\varphi$ and $\psi$  are well-defined. 
  
We have 
\begin{equation}\label{eq3}\varphi\circ \psi (y)=\varphi(a_{\mu}^{-1}a_{\nu}(y,l,x))=a_{\mu}^{-1}a_{\nu}a_{\mu}a_{\nu}^{-1}y=y.  
 \end{equation}  
\begin{equation}
\label{eq4}
\psi(\varphi((y,l,x)))=\psi (a_{\mu}a_{\nu}^{-1}y)=a_{\mu}a_{\nu}^{-1}a_{\mu}^{-1}a_{\nu}(y,l,x)=(y,l,x).
\end{equation}   
   From \eqref{eq3} and \eqref{eq4} one deduces that $\varphi$ is bijective.
 
 Claim: For every pair $(\mu,\nu)\in (F(E))^2, r(\mu)=r(\nu)$ we have $\varphi(1_{Z_{(\mu,\nu)}}(x,0,x))=1_{Z_{(\mu,\nu)}} \varphi((x,0,x))$ 
 
\emph{Proof of the claim:} By Proposition \ref{char_act} we have 
\begin{equation}\label{eq5}
\varphi(1_{Z_{(\mu,\nu)}}(x,0,x))=\begin{cases}\varphi((\mu p,|\mu|-|\nu|,x))=a_{\mu}a_{\nu}^{-1}\mu p\quad \text{if } x=\nu p\\ 0\quad \text{else}.\end{cases}
\end{equation}

On the other hand, \begin{multline}\label{eq6}
1_{Z_{(\mu,\nu)}} \varphi((x,0,x))=1_{Z_{(\mu,\nu})} x=\sigma_{\underline{a}}(\mu\nu^*).x=a
_{\mu}a_{\nu}^{-1}\mu\nu^*.x=\begin{cases}a
_{\mu}a_{\nu}^{-1}\mu p \quad \text{if }x=\nu p\\0\quad\text{else.}\end{cases}
\end{multline}
From \eqref{eq5} and \eqref{eq6} the claim is proved.

Let $\alpha,\beta\in F(E)$. We have
\begin{multline}\label{eq7}
\varphi(1_{Z_{(\alpha,\beta)}}(y,l,x))=\varphi (1_{Z_{(\alpha,\beta)}}1_{Z_{(\mu,\nu)}}((x,0,x)))=\varphi(1_{Z_{(\alpha,\beta)}Z_{(\mu,\nu)}}(x,0,x))=\\=1_{Z_{(\alpha,\beta)}Z_{(\mu,\nu)}}\varphi((x,0,x))=1_{Z_{(\alpha,\beta)}}1_{Z_{(\mu,\nu)}}\varphi((x,0,x))=1_{Z_{(\alpha,\beta)}}(y,l,x),
\end{multline} since $Z_{(\alpha,\beta)}Z_{(\mu,\nu)}$ is also a compact bisection of $G_E.$
From \eqref{eq7} one deduces that $\varphi$ is an  $A_K(G_E)$-homomorphism. Hence it is an $A_K(G_E)$-isomorphism.
\end{proof}
From Proposition \ref{triv} and Proposition \ref{noniso} we obtain results as in \cite{Chen}:
\begin{corollary}\label{cor}\begin{enumerate}

\item For $q,q'\in X \setminus E^\infty_{rat.}$  $V_{[q],K}\simeq V_{[q'],K}$ if and only if $[q]=[q'].$
\item If $x\in X \setminus E^{\infty}_{rat.}$, then $V^{\underline{a}}_{[x],K}\simeq V^{\underline{b}}_{[x],K}$, for all $\underline{a},\underline{b}\in (K^*)^{E^1}.$

\end{enumerate}
\end{corollary}
\subsection{The case of rational paths}
Let $x$ be a rational path. Following Remark \ref{isom} we may suppose $x=c^\infty,$ where $c=e_1...e_n$ is a simple closed path. For each $i\in\{1,2,...,n\}$ we define $c_i=e_{i+1}...e_{n}e_1...e_i$ (the $i$-th rotation of $c$).

By Lemma \ref{iso} $G_{\infty}=\{(c^{\infty},ln,c^{\infty})|l\in \Z\}\simeq n\Z\simeq \Z$, hence $KG_{c^\infty}\simeq K\Z\simeq K[t,t^{-1}]$. Each simple $K[t,t^{-1}]$-module is isomorphic to $K[t,t^{-1}]/I$, where $I$ is some maximal ideal of $K[t,t^{-1}]$. Since $K[t,t^{-1}]$ is a principal ideal domain, we have $I=(f(t))$(the ideal of $K[t,t^{-1}]$ generated by an irreducible polynomial $f(t)\in K[t], f\neq t$). Therefore
the spectral simple module in this case is 

$Ind_{c^\infty}(K[t,t^{-1}]/(f(t)))=KL_{c^{\infty}}\otimes_{K[t, t^{-1}]}(K[t,t^{-1}]/(f(t))).$
 
 At first we give a general result and then use it to study $Ind_{c^\infty}(K[t,t^{-1}]/(f(t))).$

\begin{lemma}\label{twist_iso}
Let $K'\supseteq K$ be any field extension.
Let $a\in K^{'*}$. We denote by $K'^{(a)}$ the following $KG_{c^\infty}$-module: $K'^{(a)}=K'$ as vector spaces and the action is $(x,ln,x)k'=a^lk', \quad \forall k'\in K', l\in \Z.$

 Let $\underline{a}\in (K'^*)^{E_1}, a_c=a_{e_1}...a_{e_n}=a,\quad V^{\underline{a}}_{[c^{\infty}],K'}$ be the $\underline{a}$-twisted Chen module over $L_{K'}(E)$. We denote by $V^{\underline{a}}_{[c^{\infty}],K'}|_K$ the restriction of $V^{\underline{a}}_{[c^{\infty}],K'}$ from $L_{K'}(E)$ to $L_K(E).$ Then we have

$Ind_{c^\infty}(K'^{(a)})=KL_{c^{\infty}}\otimes_{K[t, t^{-1}]}K'^{(a)}\simeq V^{\underline{a}}_{[c^{\infty}],K'}|_K$ as $A_K(G_E)$-modules.
\end{lemma}
\begin{proof}
Consider the $K$-bilinear map $\phi: KL_{c^{\infty}}\times K'^{(a)}\rightarrow V^{\underline{a}}_{[c^{\infty}],K'|_K}$ defined by 
$\phi((y,m,c^{\infty}),k')=a_\mu a_\nu^{-1}k'y$, where $\mu,\nu\in F(E)$ such that $y=\mu p, c^{\infty}=\nu p, |\mu|-|\nu|=m.$

We show that 
\begin{enumerate}

\item $\phi$ is well-defined. 
In fact, let $\mu_1,\mu_2,\nu_1,\nu_2\in F(E), p_1,p_2\in E^{\infty}_{rat.}$ such that  $y=\mu_1p_1=\mu_2p_2,c^{\infty}=\nu_1p_1=\nu_2p_2,|\mu_1|-|\nu_1|=|\mu_2|-|\nu_2|=m.$ 

Suppose $\mu_1=\mu_2q\implies p_2=qp_1\implies \nu_1p_1=\nu_2qp_1.$

Moreover,  $|\mu_1|-|\nu_1|=|\mu_2|-|\nu_2|\implies |\nu_1|-|\nu_2|=|q|\implies \nu_1=\nu_2q.$

Therefore, $a_{\mu_1}a_{\nu_1}^{-1}=a_{\mu_2q}a_{\nu_2q}^{-1}=a_{\mu_2}a_qa_{\nu_2}a_q^{-1}=a_{\mu_2}a_{\nu_2}^{-1}.$

\item $\phi$ is a $KG_{c^{\infty}}$-balanced product.

In fact, Let $y=\mu p, c^\infty=\nu p,|\mu|-|\nu|=m,$ it is easy to see that $p=c_{i}^\infty,$ for some $i\in\{1,2,...,n\}.$  If $l\geq 0$, we have
 \begin{equation}
\label{eq9}
\phi((y,m,c^\infty)(c^\infty,ln,c^\infty),k')=\phi((y, m+ln,c^{\infty}),k')=\phi((\mu c_i^lc_i^\infty,m+ln,\nu c_i^\infty),k')=a_{\mu c_i^l}a_{\nu}^{-1}k'y=a_{\mu}a^la_{\nu}^{-1}k'y.
\end{equation} 
since $a_{c_i}=a_{e_{i+1}}...a_{e_n}a_{e_1}...a_{e_i}=a_c=a.$
On the other hand,
\begin{equation}\label{eq10}
\phi((y,m,c^\infty),(c^\infty,ln,c^\infty)k')=\phi((y,m,c^\infty),a^l k')=a_{\mu}a_{\nu}^{-1}a^lk'y.
\end{equation}
From \eqref{eq9} and \eqref{eq10} we see that $\phi$ is a $KG_{c^\infty}$-balanced product.
\end{enumerate}
By the universal property of the tensor product, there exists the unique $K$-linear map $\varphi: KL_{c^\infty}\otimes_{KG_{c^\infty}} K'^{(a)}\rightarrow V^{\underline{a}}_{[c^{\infty}],K'|_K}$ such that

$\varphi((y,m,c^\infty)\otimes k')=\phi((y,m,c^\infty),k')=a_{\mu}a_{\nu}^{-1}k'y,$ where $y=\mu c_i^\infty, c^\infty=\nu c_{i}^\infty, |\mu|-|\nu|=m.$

Consider the $K$-linear map $\psi: V^{\underline{a}}_{[c^{\infty}],K'}|_K\rightarrow KL_{c^\infty}\otimes_{KG_{c^\infty}} K'^{(a)}$ defined by

$\psi(y)=(y,|\mu|-|\nu|,c^\infty)\otimes a_{\mu}^{-1}a_{\nu}.$

We show that $\psi$ is well-defined. In fact, if $y=\mu_1c_{i_1}^\infty=\mu_2c_{i_2}^\infty, c^\infty=\nu_1c_{i_1}^\infty=\nu_2 c_{i_2}^\infty, |\mu_1|\geq |\mu_2|$. Suppose $\mu_1=\mu_2q$. We have $c_{i_2}^\infty=qc_{i_1}^\infty\implies q=e_{i_2+1}...e_{i_n}e_1...e_{i_1}c_{i_1}^{l'}.$ 
Let $\nu_1=c^{l_1}e_1...e_{i_1}, \nu_2=c^{l_2}e_1...e_{i_2}.$

We have $|\mu_1|-|\mu_2|+|\nu_2|-|\nu_1|=n-i_2+i_1+l'n+(l_2-l_1)n+i_2-i_1=(l'+1+l_2-l_1)n.$ 

Therefore
\begin{multline}
(y,|\mu_1|-|\nu_1|,c^\infty)\otimes a_{\mu_1}^{-1}a_{\nu_1}=\\=(y,|\mu_2|-|\nu_2|,c^{\infty})(c^{\infty},|\mu_1|-|\nu_1|-|\mu_2|+|\nu_2|, c^{\infty})\otimes a_{\mu_2}^{-1}a_{e_{i_2+1}}^{-1}...a_{e_{i_n}}^{-1}a_{e_1}^{-1}...a_{e_{i_1}}^{-1}a^{-l'}a^{l_1}a_{e_1}...a_{e_{i_1}}=\\=(y,|\mu_2|-|\nu_2|, c^\infty)\otimes a^{l'+1+l_2-l_1}a_{\mu_2}^{-1}a_{e_{i_2+1}}^{-1}...a_{e_n}^{-1}a^{l_1-l'}=(y,|\mu_2|-|\nu_2|, c^\infty)\otimes a_{\mu_2}^{-1}a_{e_{i_2+1}}^{-1}...a_{e_n}^{-1}a^{l_2+1}=\\=(y,|\mu_2|-|\nu_2|, c^\infty)\otimes a_{\mu_2}^{-1}a^{l_2}a_{e_{i_2+1}}^{-1}...a_{e_n} ^{-1}a_{e_1}a_{e_2}...a_{e_n}=(y,|\mu_2|-|\nu_2|, c^\infty)\otimes a_{\mu_2}^{-1}a^{l_2}a_{e_1}...a_{e_{i_2}}=\\=(y,|\mu_2|-|\nu_2|, c^\infty)\otimes a_{\mu_2}^{-1}a_{\nu_2}.
\end{multline}
We have
\begin{equation}\label{eq11}\varphi\circ\psi(y)=\varphi((y,|\mu|-|\nu|,c^{\infty}\otimes a_{\mu}^{-1}a_{\nu})=a_{\mu}a_{\nu}^{-1}a_{\mu}^{-1}a_{\nu}y=y.
\end{equation}
\begin{equation}\label{eq12}
\psi\circ \varphi((y,m,c^\infty)\otimes k')=\psi(a_{\mu}a_{\nu}^{-1}k'y)=(y,m,c^\infty)\otimes a_{\mu}^{-1}a_\nu a_{\mu}a_{\nu}^{-1}k'=(y,m,c^\infty)\otimes k'.
\end{equation}
From \eqref{eq11} and \eqref{eq12} one deduces that $\varphi$ is bijective.

Claim:For all $ \mu,\nu\in F(E),k'\in K'$ we have
\begin{equation}\label{eq13}\varphi(1_{Z_{(\mu,\nu)}}((c^\infty,0,c^\infty)\otimes k'))=1_{Z_{(\mu,\nu)}}\varphi ((c^\infty,0,c^\infty)\otimes k') \end{equation}
\emph{Proof of the claim:} The left hand side is 
$\varphi((\mu c_i^\infty,|\mu|-|\nu|,c^\infty)\otimes k')=a_{\mu}a_{\nu}^{-1}k'\mu c_i^\infty.$

The right hand side is $1_{Z_{(\mu,\nu)}}k'c^\infty=\sigma_{\underline{a}}(\mu\nu^*).(k'c^\infty)=a_{\mu}a_{\nu}^{-1}k'\mu c_i^\infty.$

Hence by using \eqref{eq12} similarly as in the last part of the proof of Proposition \ref{triv} one deduces that $\varphi$ is an $A_K(G_E)$-isomorphism. 
\end{proof}

When $K'=K$, from Lemma \ref{twist_iso} we have immediately
\begin{corollary}\label{cor1}

$Ind_{c^\infty}(K^{(a)})\simeq V^{\underline{a}}_{[c^{\infty}],K}$

\end{corollary}

From Lemma \ref{cor} and Proposition \ref{res_in} one can obtain the following results as in \cite{Chen}.
\begin{corollary}\label{cor2}

$V^{\underline{a}}_{[c^{\infty}],K}\simeq V^{\underline{b}}_{[c^{\infty}],K}$ if and only if $a_c=b_c$ (i.e. $ab^{-1}$ is $c$-stable.)

\end{corollary}

When $K'=K[t,t^{-1}]/(f(t)),a=\bar{t}, f$ is an irreducible polynomial in $K[t], f\neq t$, since  
 $K'=K[t,t^{-1}]/(f(t))$ is a simple $K[t,t^{-1}]$-module, from Lemma \ref{twist_iso} and Theorem \ref{simp} we have immediately the following result as in \cite{AKR}
\begin{corollary}\label{cor3} $V^{f}_{[c^\infty],K}\simeq Ind_{c^\infty}(K[t,t^{-1}]/(f(t)))$ as simple $A_K(G_E)$-modules, where $V^{f}_{[c^\infty],K}=V^{\bar{t}}_{[c^\infty],K'}|_K.$\end{corollary}
\begin{remark}\label{rem1} For $a\in K^{*},\underline{a}\in (K^*)^{E_1}$ such that $a_c=a$ we have
$V^{t-a}_{[c^\infty, K]}\simeq V^{\underline{a}}_{[c^\infty], K}$ as $A_K(G_E)$-modules. 

In particular, $V^{x-1}_{[c^\infty, K]}\simeq V_{[c^\infty], K}$ by Remark \ref{rem2}.
\end{remark}
\begin{proof}
Consider $\theta: K[t,t^{-1}]\rightarrow K^{(a)}$ defined by $\theta(g)=g(a)$. Then $\theta$ is surjective. In fact, if $b\in K$, then $\exists g=t-b+a$ such that $g(a)=b.$

Moreover, $\theta((c^\infty,ln,c^\infty)g)=\theta(t^lg)=a^lg(a)=(c^\infty,ln,c^\infty)g(a)(c^\infty,ln,c^\infty)\theta(g)$ by Definition of $K^{(a)}$. Hence $\theta$ is a surjective $KG_{c^\infty}$-homomorphism. Therefore by theorem of module homomorphisms: 

$K[t,t^{-1}]/Ker (\theta)\simeq K^{(a)}$ as $KG_{c^\infty}$-modules.

$Ker(\theta)=\{g\in K[t,t^{-1}]|g(a)=0\}=(t-a)\implies K[t,t^{-1}]/(t-a)\simeq K^{(a)}$ as $KG_{c^\infty}$-modules.

Hence $KL_{c^\infty}\otimes_{KG_{c^\infty}}(K[t,t^{-1}]/(t-a))\simeq KL_{c^\infty}\otimes_{KG_{c^\infty}}K^{(a)}.$ Therefore by Corollories \ref{cor1} and \ref{cor3} $V^{t-a}_{[c^\infty, K]}\simeq V^{\underline{a}}_{[c^\infty], K}$ as $A_K(G_E)$-modules.

In particular, $V^{t-1}_{[c^\infty, K]}\simeq V_{[c^\infty], K}$ by Remark \ref{rem2}.
\end{proof}
\begin{remark}$[x]=\{y\in X|y\sim x\}\implies \alpha=(y,m,x)\in G_E$ for some $m$. By Definition of $G_E$: $r(\alpha)=x,s(\alpha)=y\implies$ $y$ is in the orbir of $x$.\end{remark}
Following Theorem \ref{specmod},Lemma \ref{noniso}, Proposition \ref{triv} and Corollaries \ref{cor},\ref{cor1},\ref{cor2}, Remark \ref{rem1} we now can classify all spectral simple modules over Leavitt path algebras
\begin{theorem} Let $E$ be a digraph, $K$ be a field.  Let $D\subseteq X$ be a subset containing exactly one element from each orbit. 
Then the set

\begin{multline}\label{set}\{V_{[x],K}|x\text{ is an infinite path }, x\in D\}\cup \{V_{[w],K}|w \text{ is not regular}\}\cup\\\cup \{V^f_{[c^\infty], K}|f \text{ is an irreducible polynomial in } K[t]|f\neq t-1, f\neq t, c \text{ is a simple closed path in E}\}\end{multline}

forms the full list of all pairwise nonisomorphic  simple spectral  modules over $L_K(E).$

In particular, all finite dimensional simple module over $L_K(E)$ are the modules in the list \eqref{set} with $x$'s which have finite orbits.
\end{theorem}

\bibliographystyle{plain}
\bibliography{bibliografia1}

\end{document}